\documentclass[11pt]{article} 

\usepackage[lmargin=1in,rmargin=1in,tmargin=1in,bmargin=1in]{geometry}
\usepackage[ps,all,arc,rotate]{xy}
\usepackage{graphicx, float, epstopdf}
\usepackage{hyperref, color}
\usepackage{centernot}
\usepackage{fancyhdr}
\usepackage[utf8]{inputenc}
\usepackage{amsfonts,amssymb,amsmath,amsthm,mathrsfs}
\usepackage{graphics, setspace}
\usepackage{braket}
\usepackage{mathtools}
\usepackage{tikz}
\usepackage{enumerate}

\numberwithin{equation}{section}
\numberwithin{figure}{section}
  
\allowdisplaybreaks[4]         

\newtheorem{lemma}{Lemma}[section]
\newtheorem{theorem}[lemma]{Theorem}
\newtheorem{corollary}[lemma]{Corollary}
\newtheorem{proposition}[lemma]{Proposition}

\newtheorem{remark}[lemma]{Remark}

\newcommand{\bZ}{\mathbb{Z}}

\newcommand{\bQ}{\mathbb{Q}}
\newcommand{\bA}{\mathbb{A}}
\newcommand{\bN}{\mathbb{N}}
\newcommand{\Sym}{\text{Sym}}

\newcommand{\ve}{\varepsilon}

\newcommand{\bR}{\mathbb{R}}
\newcommand{\bC}{\mathbb{C}}

\newcommand{\cH}{\mathcal{H}}

\def\mo{\operatorname{mod}}

\usepackage[
backend=biber,
style=alphabetic,
]{biblatex}
\begin{document}
\title{On M\"obius functions from automorphic forms and a generalized Sarnak's conjecture}
\author{Zhining Wei, Shifan Zhao}	

\maketitle

\begin{abstract}
   In this paper, we consider M\"obius functions associated with two types of $L$-functions: Rankin-Selberg $L$-functions of symmetric powers of distinct holomorphic cusp forms and $L$-functions of Maass cusp forms. We show that these M\"obius functions are weakly orthogonal to bounded sequences. As a direct corollary, a generalized Sarnak's conjecture holds for these two types of M\"obius functions.
   
\end{abstract}

\section{Introduction}
Let $a(n), b(n)$ be two arithmetic functions. We say $a(n)$ is \textit{(asymptotically) orthogonal} to $b(n)$, if 
$$\lim_{N \to \infty}\frac{1}{N}\sum_{n=1}^Na(n)b(n) = 0.$$
\par
Denote by $\mu(n)$ the M\"obius function. The Sarnak's conjecture, proposed by P. Sarnak (see \cite{S1},\cite{S2}), predicts that $\mu(n)$ behaves so randomly that it should be orthogonal to any deterministic sequence $\xi(n)$. Precisely, a sequence $\xi(n)$ is called \textit{deterministic} if it has the form
$$\xi(n) = f(T^nx),$$
where $T: X \to X$ is a continuous map on a compact topological space $X$ such that the system $(X,T)$ has zero entropy, $f:X \to \bC$ is a continuous function, and $x \in X$. 
\par
As examples, we state four cases where Sarnak's conjecture has been verified. It should be noted that the first three results stated below were proved before Sarnak's conjecture was formally proposed.
\par
If we take $\xi(n)$ to be the constant sequence $(1,1,\dots)$, then the orthogonality relation reads
$$\lim_{N \to \infty} \frac{1}{N} \sum_{n=1}^N \mu(n) = 0,$$
which is known to be equivalent to the Prime Number Theorem.
\par
The case where $\xi(n)$ is an additive character was studied by H. Davenport in \cite{D}. Precisely, let $\alpha$ be any real number and $A > 0$. It was proved in \cite{D} that
$$\frac{1}{N} \sum_{n=1}^N \mu(n)e(\alpha n) = O_A\left(\frac{1}{\log^AN}\right),$$
where the implied constant depends on $A$ and is uniform in $\alpha$. In this case, $\mu(n)$ is said to be \textit{strongly orthogonal} to $e(\alpha n)$, following terminologies introduced in \cite{GT}.
\par
B. Green and T. Tao considered the case where $\xi(n)$ is a nilsequence in \cite{GT}. What they proved is that for any nilmanifold $G/\Gamma$, any polynomial sequence $g:\bZ \to G$, and a Lipschitz function $F:G/\Gamma \to \bR$, one has
$$\frac{1}{N}\sum_{n=1}^N \mu(n)F(g(n)\Gamma) = O_{G,\Gamma,F,A}\left(\frac{1}{\log^AN}\right)$$
for any $A > 0$. This implies Davenport's result, with a suitable choice of $G,\Gamma,g$ and $F$.
\par
The case where $X$ is a compact abelian group and $T:X \to X$ is an affine linear map was studied by J. Liu and P. Sarnak. It was proved in \cite{LS2015} that for such a flow $(X,T)$, Sarnak's conjecture is true.
\par
Another worth mentioning result is due to \'E. Fouvry and S. Ganguly, where Fourier coefficients of modular forms are considered. Let $f$ be a Maass cusp form with normalized Fourier coefficients $\nu_f(n)$, and let $\alpha$ be any real number. Then it was proved in \cite{FG} that there exists some absolute effective constant $c_0$ such that
$$\frac{1}{N}\sum_{n=1}^N \mu(n)\nu_f(n)e(\alpha n) = O_f(e^{-c_0\sqrt{\log N}}),$$
which implies strong orthogonality between $\mu(n)\nu_f(n)$ and the additive character $e(\alpha n)$. This can be viewed as a GL(2) analogue of Davenport's result.
\par
One way to view the M\"obius function $\mu(n)$ is by the following identity
$$\zeta(s)^{-1} = \sum_{n=1}^\infty \frac{\mu(n)}{n^s},\hspace{3mm} \Re(s)>1,$$
where $\zeta(s)$ denotes the Riemann zeta function. One may replace $\zeta(s)$ by any $L$-function $L(s,\pi)$ (see section 5.1 of \cite{IwKo2004} for what we mean by an $L$-function), and express its reciprocal as a Dirichlet series
$$L(s,\pi)^{-1} = \sum_{n=1}^\infty \frac{\mu_\pi(n)}{n^s},\hspace{3mm}, \Re(s) \gg 1$$
where we call $\mu_\pi(n)$ the M\"obius function associated with $L(s,\pi)$.
\par
A natural question to ask is whether these generalized M\"obius functions are orthogonal to deterministic sequences, or in other words, whether a generalized Sarnak's conjecture is ture for them. A recent result towards this direction is due to Y. Jiang and G. L\"u. Precisely, let $F$ be a Hecke-Maass form for SL($2,\bZ$) or SL($3,\bZ$), and let $\alpha$ be a real number. Then it was proved in \cite{JL} that
$$\frac{1}{N} \sum_{n=1}^N \mu_F(n)e(\alpha n) = O_F(e^{-c\sqrt{\log N}})$$
for some absolute constant $c>0$. This result implies that $\mu_F(n)$ is strongly orthogonal to additive characters $e(\alpha n)$.
\par
In this paper we consider two types of $L$-functions and their associated M\"obius functions. The first type is Rankin-Selberg $L$-functions of two symmetric powers of holomorphic cusp forms. The second type is $L$-functions of Maass cusp forms. For a detailed description of these $L$-functions, see Section 2.
\par
Our first main result is that generalized Sarnak's conjecture holds for M\"obius functions associated to $L(s,\Sym^{m_1}(f) \times \Sym^{m_2}(g))$, where $f,g$ are distinct holomorphic cusp forms, unless $m_1 = m_2 = 0$ (in that case, $L(s,\Sym^{m_1}(f) \times \Sym^{m_2}(g))$ reduces to $\zeta(s)$). In fact, we will prove a stronger result:
\begin{theorem}\label{absolute rankin selberg}
Let $f,g$ be two distinct holomorphic Hecke eigenforms for SL$(2,\bZ)$. Let $m_1,m_2 \geq 0$ be natural numbers, not both zero. Then there exists some constant $\eta_{m_1,m_2} > 0$, depending only on $m_1$ and $m_2$ such that 
$$\frac{1}{N}\sum_{n=1}^N|\mu_{\Sym^{m_1}(f) \times \Sym^{m_2}(g)}(n)| = O\left(\frac{1}{\log^{\eta_{m_1,m_2}}N}\right).$$
\end{theorem}

Since deterministic sequences are bounded, this theorem immediately implies the following

\begin{corollary}\label{weak orthogonality rankin selberg}
Let $f,g$ be two distinct holomorphic Hecke eigenforms for SL$(2,\bZ)$. Let $m_1,m_2 \geq 0$ be natural numbers, not both zero. Let $\xi(n)$ be any deterministic sequence. Then there exists some constant $\eta_{m_1,m_2} > 0$, depending only on $m_1$ and $m_2$ such that 
$$\frac{1}{N}\sum_{n=1}^N\mu_{\Sym^{m_1}(f) \times \Sym^{m_2}(g)}(n)\xi(n) = O_\xi\left(\frac{1}{\log^{\eta_{m_1,m_2}}N}\right),$$
where the implied constant depends only on the sup-norm of $\xi$.
\end{corollary}

\begin{remark}
Following terminologies in \cite{GT}, Theorem \ref{absolute rankin selberg} implies that these M\"obius functions are weakly orthogonal to bounded sequences and Corollary \ref{weak orthogonality rankin selberg} states that $\mu_{\Sym^{m_1}(f) \times \Sym^{m_2}(g)}(n)$ is weakly orthogonal to any deterministic sequence $\xi(n)$.
\end{remark}

\begin{remark}
In the case $m_1=m_2=0$, one has the following well-known asymptotic formula
$$\sum_{n=1}^N |\mu(n)| \sim \frac{N}{\zeta(2)},\hspace{3mm} N \to \infty.$$
\end{remark}

For Maass forms, our result is the following

\begin{theorem}\label{absolute maass}
Let $\phi$ be a Hecke Maass form for SL$(2,\bZ)$. Then we have
$$\frac{1}{N}\sum_{n=1}^N|\mu_\phi(n)| = O\left(\frac{1}{\log^{1/12}N}\right).$$
\end{theorem}

\begin{corollary}
Let $\phi$ be a Hecke Maass form for SL$(2,\bZ)$. Let $\xi(n)$ be any deterministic sequence. Then we have
$$\frac{1}{N}\sum_{n=1}^N\mu_\phi(n)\xi(n) = O_{\xi}\left(\frac{1}{\log^{1/12}N}\right),$$
where the implied constant depends only on the sup-norm of $\xi$.
\end{corollary}

\section{Preliminaries}
Let $f$ be a holomorphic Hecke eigenform for SL$(2,\bZ)$ of weight $k$. This is a holomorphic function on the upper half plane $\cH = \{z = x+iy:x \in \bR, y> 0\}$ with a Fourier expansion
$$f(z) = \sum_{n=1}^\infty \lambda_f(n)n^\frac{k-1}{2}e(nz),\hspace{3mm} \lambda_f(1) = 1.$$
The Hecke $L$-function of $f$ is defined by
$$L(s,f) = \sum_{n=1}^\infty \frac{\lambda_f(n)}{n^s},\hspace{3mm} \Re(s) > 1,$$
which admits an analytic continuation to $\bC$ with a functional equation $s \to 1-s$. By Hecke theory, $L(s,f)$ has a Euler product expansion
$$L(s,f) = \prod_p \left(1-\frac{\alpha_f(p)}{p^s}\right)^{-1}\left(1-\frac{\beta_f(p)}{p^s}\right)^{-1},\hspace{3mm} \Re(s) > 1,$$
where $\alpha_f(p)\beta_f(p) = 1$ and $\alpha_f(p) + \beta_f(p) = \lambda_f(p)$.
Let $m \geq 0$ be a natural number. The symmetric $m^{th}$ power $L$-function of $f$ is defined by
$$L(s,\Sym^m(f)) = \prod_p\prod_{i=0}^m \left(1-\frac{\alpha_f(p)^i\beta_f(p)^{m-i}}{p^s}\right)^{-1},\hspace{3mm} \Re(s) > 1.$$
It was recently proved by J. Newton and J. Thorne \cite{NT} that $L(s,\Sym^m(f))$ comes from an cuspidal automorphic representation of GL$(m+1,\bA_\bQ)$, denoted by $\Sym^m(f)$. Thus for two forms $f,g$ and two natural numbers $m_1,m_2$, the Rankin-Selberg $L$-function 
$$L(s,\Sym^{m_1}(f) \times \Sym^{m_2}(g)) = \prod_p \prod_{i=0}^{m_1}\prod_{j=0}^{m_2} \left(1-\frac{\alpha_f(p)^i\beta_f(p)^{m_1-i}\alpha_g(p)^j\beta_g(p)^{m_2-j}}{p^s}\right)^{-1}$$
has a meromorphic continuation to $\bC$, with a functional equation $s \to 1-s$. It has a simple pole at $s=1$ if and only if $\Sym^{m_1}(f) \cong \Sym^{m_2}(g)$ as automorphic representations.
\par
Let $\phi$ be Hecke Maass form for SL$(2,\bZ)$ with Laplacian eigenvalue $\lambda = \frac{1}{4}+t^2$. This is a real-analytic function on the upper half plane $\cH$, which has a Fourier expansion
$$\phi(z) = \sqrt{y}\sum_{n \neq 0}\lambda_\phi(n)K_{it}(2\pi|n|y)e(nx)$$
where $z = x+iy$ and $K_{it}$ is the $K$-Bessel function. The $L$-function of $\phi$ is defined as
$$L(s,\phi) = \sum_{n=1}^\infty \frac{\lambda_\phi(n)}{n^s},\hspace{3mm} \Re(s)>1.$$
It also has an analytic continuation to $\bC$ with a functional equation $s \to 1-s$, and an Euler product expansion. Symmetric powers of Maass forms are defined in the same way as in the holomorphic case. But $\Sym^j(\phi)$ are proved to be cuspidal only for $j \leq 4$ (see \cite{GelbartJacquet1978},\cite{KimShahidi2000} and \cite{Kim2003}).
\par
We now recall two operations on $L$-functions. For a general notation of $L$-functions, see \cite{IwKo2004} Section 5.1. Let $L(s,\pi_1), L(s,\pi_2)$ be two $L$-functions of degree $d_1,d_2$ respectively with Dirichlet series expansions
$$L(s,\pi_i) = \sum_{n=1}^\infty \frac{\lambda_{\pi_i}(n)}{n^s},\hspace{3mm} \Re(s)>1$$
and Euler product expansions
$$L(s,\pi_i) = \prod_p \prod_{j=1}^{d_i}\left(1-\frac{\alpha_{\pi_i,j}(p)}{p^s}\right)^{-1},\hspace{3mm} \Re(s)>1$$
where $i=1,2$. We set
$$L(s,\pi_1 \boxplus \pi_2) = L(s,\pi_1)L(s,\pi_2)$$
and
$$L(s,\pi_1 \times \pi_2) = \prod_p\prod_{j_1=1}^{d_1}\prod_{j_2=1}^{d_2}\left(1-\frac{\alpha_{\pi_1,j_1}(p)\alpha_{\pi_2,j_2}(p)}{p^s}\right)^{-1},\hspace{3mm} \Re(s) \gg 1.$$
In particular, if $\pi_1, \pi_2$ are everywhere unramified cuspidal representations on GL$(d_1,\bA_\bQ)$ and GL$(d_2,\bA_\bQ)$ respectively, then $\pi_1 \boxplus \pi_2$ is their isobaric sum, and $L(s,\pi_1 \times \pi_2)$ is the Rankin-Selberg $L$-function of $\pi_1$ and $\pi_2$. For a prime $p$, the following relations hold:
\begin{align*}
    \lambda_{\pi_1 \boxplus \pi_2}(p) &= \lambda_{\pi_1}(p) + \lambda_{\pi_2}(p), \\
    \lambda_{\pi_1 \times \pi_2}(p) &= \lambda_{\pi_1}(p)\lambda_{\pi_2}(p).
\end{align*}
\par
We then collect some lemmas that will be used in the proof of Theorem \ref{absolute rankin selberg} and Theorem \ref{absolute maass}.
\par
The first lemma is due to P. Shiu (\cite[Theorem~1]{S}), which enables us to reduce a sum of certain multiplicative functions over positive integers to a sum over prime numbers.
\begin{lemma}\label{Shiu}
    Let $f:\bN \to \bC$ be a non-negative multiplicative function with the following properties
    \begin{itemize}
        \item There exists some $A_1 > 0$ such that for any prime $p$ and $l \geq 1$, $f(p^l) \leq A_1^l.$
        \item For any $\ve > 0$, $f(n) \ll_\ve n^\ve.$
    \end{itemize}
    Let $0 <a <k$ be integers with $(a,k)=1$. Let $0 < \alpha,\beta <\frac{1}{2}$ be real numbers. Then
    $$\sum_{\substack{x-y<n\leq x\\ n \equiv a(\mo k)}}f(n) \ll \frac{y}{\phi(k)} \frac{1}{\log x} \exp\left(\sum_{p \leq x, p \nmid k}\frac{f(p)}{p}\right)$$
    uniformly in $a,k,y$, provided $k < y^{1-\alpha}$ and $x^\beta < y \leq x.$
\end{lemma}

\par
Next, we state an inequality proved by H. Tang and J. Wu (\cite[Lemma~3.3]{TaWu2016})
\begin{lemma}\label{inequality}
Let $m \geq 1$ be a positive integer. Set
\begin{align*}
    \kappa_m &= (m+2)(m+1), \\
    a_0(m) &= \frac{(\kappa_m-3)\sqrt{\kappa_m}+2}{2(\kappa_m-1)^2}, \\
    a_1(m) &= \frac{(\kappa_m^2+3)\sqrt{\kappa_m}-4\kappa_m}{2(\kappa_m-1)^2}, \\
    a_2(m) &= -\frac{(\kappa_m^2+\kappa_m)\sqrt{\kappa_m}-2\kappa_m^2}{2(\kappa_m-1)^2}.
\end{align*}
Then for $0 \leq t \leq 1$ we have
$$\sqrt{t} \leq a_0(m) + a_1(m)t + a_2(m)t^2.$$
\end{lemma}

\par
The next lemma gives a decomposition formula of Rankin-Selberg $L$-functions of two symmetric powers of the same cusp form.

\begin{lemma}\label{decomposition}
Let $f$ be a holomorphic Hecke eigenform for SL$(2,\bZ)$. Let $m,r \geq 0$ be two natural numbers. Then we have
$$L(s,\Sym^m(f) \times \Sym^{m+r}(f)) = \prod_{i=0}^m L(s,\Sym^{2i+r}(f)). $$
\end{lemma}

\par
\begin{proof}This lemma can be directly checked by comparing Satake parameters of both sides at each prime $p$. A possible way is to fix $r$ and perform an induction on $m$.
\end{proof}

We need the following orthogonality relations for $L$-functions. 

\begin{lemma}\label{orthognality relations for L functions}
Let $\pi_1,\pi_2$ be two self-dual, automorphic cuspidal representations of GL$(d_1,\bA_\bQ)$ and GL$(d_2,\bA_\bQ)$ respectively. Suppose that one of the following conditions are satisfied:
\begin{enumerate}[(a)]
\item 
The Ramanujan conjecture holds for at least one of $\pi_1,\pi_2.$
\item 
We have $d_1,d_2\leq 4.$
\end{enumerate}

Then 
\[\sum_{p\leq x}\frac{\lambda_{\pi_1}(p)\lambda_{\pi_2}(p)}{p}=
\left\{\begin{array}{cc}
\log\log x+O(1)&\mbox{if $\pi_1=\pi_2$}\\
O(1)&\mbox{otherwise.}
\end{array}\right.
\]
\end{lemma}

\begin{proof}
For the proof, one can refer to \cite{RudnickSarnak1996}, \cite{LiuWangYe2005} and \cite{WangWeiYanYi2023}. For (a), one can use \cite[Proposition~3.1]{WangWeiYanYi2023}. For (b), one can use \cite[Corollary~1.5]{LiuWangYe2005}.
\end{proof}

\par
The next lemma tells us when two symmetric power representations of distinct holomorphic forms are isomorphic. 

\begin{lemma}\label{multiplicity one}
Let $f,g$ be two distinct holomorphic Hecke eigenforms for SL$(2,\bZ)$. Let $m_1,m_2 \geq 0$ be two natural numbers. Then $\Sym^{m_1}(f) \cong \Sym^{m_2}(g)$ if and only if $m_1 = m_2 = 0$.    
\end{lemma}
\begin{proof}
For a proof, see \cite[Corollary 5.2]{CoMi2004}.
\end{proof}

The last lemma, due to K. Venkatasubbareddy, A. Kaur and A. Sankaranarayanan, \cite[Lemma~3.6]{VeKaSa2022}, is used to prove an analogy of Lemma \ref{Shiu}. It enables us to deal with Maass forms, where Ramanujan conjecture is not available. 

\begin{lemma}\label{Maass form Shiu}
Let $g$ be a non-negative multiplicative function. Suppose that exist non-negative functions $h_1(x), h_2(x)$, with $h_1(x)$ increasing, such that
$$\sum_{p \leq x}g(p)\log p \ll xh_1(x)$$
and
$$\sum_p \sum_{\alpha \geq 2} \frac{g(a^\alpha)}{p^\alpha}\log p^\alpha \ll h_2(x).$$
Then we have
\[\sum_{n\leq x}g(n)\ll(h_1(x)+h_2(x)+1)\frac{x}{\log x}\sum_{n\leq x}\frac{g(n)}{n}.\]
\end{lemma}

\section{Proof of Theorem 1.1}
We write $\pi = \Sym^{m_1}(f) \times \Sym^{m_2}(g)$.
To start with, we apply Lemma \ref{Shiu} to the non-negative multiplicative function $|\mu_{\pi}(n)|$, with suitable choices of $a,k,\alpha,\beta$, to get
$$\sum_{n=1}^N|\mu_{\pi}(n)| \ll \frac{N}{\log N}\exp\left(\sum_{p \leq N}\frac{|\mu_\pi(p)|}{p}\right).$$
The conditions in Lemma \ref{Shiu} are satisfied because the Ramanujan conjecture of holomorphic cusp forms is known to be true, by a remarkable result of P. Deligne, and $\mu_{\pi}(n)$ is supported on $(d_\pi+1)^{th}$ power-free integers, where $d_\pi = (m_1+1)(m_2+1)$.
Observe that 
$$\mu_\pi(p) = -\sum_{j=1}^{d_\pi} \alpha_{\pi,j}(p) = -\lambda_\pi(p).$$
So one has
$$\sum_{p \leq N}\frac{|\mu_\pi(p)|}{p} = \sum_{p \leq N}\frac{|\lambda_\pi(p)|}{p}.$$
We then apply Lemma \ref{inequality}, with $t = \left(\frac{|\lambda_\pi(p)|}{d_\pi}\right)^2$ and $m=d_\pi-1$, to get
$$|\lambda_\pi(p)| \leq a_0(d_\pi-1)d_\pi + a_1(d_\pi-1)\frac{\lambda_\pi(p)^2}{d_\pi} + a_2(d_\pi-1)\frac{\lambda_\pi(p)^4}{d_\pi^3}.$$
Thus we have
$$\sum_{p \leq N}\frac{|\lambda_\pi(p)|}{p} \leq a_0(d_\pi-1)d_\pi\sum_{p \leq N}\frac{1}{p} + \frac{a_1(d_\pi-1)}{d_\pi}\sum_{p \leq N}\frac{\lambda_\pi(p)^2}{p} + \frac{a_2(d_\pi-1)}{d_\pi^3}\sum_{p \leq N}\frac{\lambda_\pi(p)^4}{p}.$$
For the first term we use the asymptotic formula 
$$\sum_{p \leq N}\frac{1}{p} = \log \log N + O(1).$$
For the second term we note that
$$\lambda_\pi(p)^2 = \lambda_{\pi \times \pi}(p).$$
We then use Lemma \ref{decomposition} to decompose $\pi \times \pi$:
\begin{align*}
    \pi \times \pi &= (\Sym^{m_1}(f) \times \Sym^{m_2}(g)) \times (\Sym^{m_1}(f) \times \Sym^{m_2}(g)) \\
    &= (\Sym^{m_1}(f) \times \Sym^{m_1}(f)) \times (\Sym^{m_2}(g) \times \Sym^{m_2}(g)) \\
    &= (\boxplus_{i=0}^{m_1}\Sym^{2i}(f)) \times (\boxplus_{j=0}^{m_2}\Sym^{2j}(g)) \\
    &= \boxplus_{i=0}^{m_1}\boxplus_{j=0}^{m_2} \Sym^{2i}(f) \times \Sym^{2j}(g)
\end{align*}
Thus the sum in the second term is
\begin{align*}
    \sum_{p \leq N} \frac{\lambda_\pi(p)^2}{p} &= \sum_{p \leq N}\frac{\lambda_{\pi \times \pi}(p)}{p} \\
    &= \sum_{i=0}^{m_1}\sum_{j=0}^{m_2}\sum_{p \leq N}\frac{\lambda_{\Sym^{2i}(f) \times \Sym^{2j}(g)}(p)}{p} \\
    &= \sum_{i=0}^{m_1}\sum_{j=0}^{m_2}\delta_{\Sym^{2i}(f) \cong \Sym^{2j}(g)}\log \log N + O(1) \\
    &= \log \log N + O(1)
\end{align*}
in view of Lemma \ref{orthognality relations for L functions} and Lemma \ref{multiplicity one}.
For the third term, we use the relation
$$\lambda_\pi(p)^4 = \lambda_{\pi \time \pi \times \pi \times \pi \times \pi}(p)$$
and the decomposition
\begin{align*}
    \pi \time \pi \times \pi \times \pi \times \pi &= (\boxplus_{i_1=0}^{m_1}\boxplus_{j_1=0}^{m_2} \Sym^{2i_1}(f) \times \Sym^{2j_1}(g)) \times (\boxplus_{i_2=0}^{m_1}\boxplus_{j_2=0}^{m_2} \Sym^{2i_2}(f) \times \Sym^{2j_2}(g)) \\
    &= \boxplus_{i_1=0}^{m_1}\boxplus_{j_1=0}^{m_2}\boxplus_{i_2=0}^{m_1}\boxplus_{j_2=0}^{m_2}(\Sym^{2i_1}(f) \times \Sym^{2i_2}(f)) \times (\Sym^{2j_1}(g) \times \Sym^{2j_2}(g)) \\    
    &= \boxplus_{i_1=0}^{m_1}\boxplus_{j_1=0}^{m_2}\boxplus_{i_2=0}^{m_1}\boxplus_{j_2=0}^{m_2}\boxplus_{k=0}^{\min\{2i_1,2i_2\}}\boxplus_{l=0}^{\min\{2j_1,2j_2\}} \Sym^{2k+2|i_1-i_2|}(f) \times \Sym^{2l+2|j_1-j_2|}(g)
\end{align*}
This gives
\begin{align*}
    \sum_{p \leq N}\frac{\lambda_\pi(p)^4}{p} &= \sum_{p \leq N}\frac{\lambda_{\pi \time \pi \times \pi \times \pi \times \pi}(p)}{p} \\
    &= \sum_{i_1=0}^{m_1}\sum_{i_2=0}^{m_1}\sum_{j_1=0}^{m_2}\sum_{j_2=0}^{m_2}\sum_{k=0}^{\min\{2i_1,2i_2\}}\sum_{l=0}^{\min\{2j_1,2j_2\}}\sum_{p \leq N}\frac{\lambda_{\Sym^{2k+2|i_1-i_2|}(f) \times \Sym^{2l+2|j_1-j_2|}(g)}(p)}{p} \\
    &= \sum_{i_1=0}^{m_1}\sum_{i_2=0}^{m_1}\sum_{j_1=0}^{m_2}\sum_{j_2=0}^{m_2}\sum_{k=0}^{\min\{2i_1,2i_2\}}\sum_{l=0}^{\min\{2j_1,2j_2\}} \delta_{\Sym^{2k+2|i_1-i_2|}(f) \cong \Sym^{2l+2|j_1-j_2|}(g)} \log \log N + O(1) \\
    &= \sum_{i_1=0}^{m_1}\sum_{i_2=0}^{m_1}\sum_{j_1=0}^{m_2}\sum_{j_2=0}^{m_2}\sum_{k=0}^{\min\{2i_1,2i_2\}}\sum_{l=0}^{\min\{2j_1,2j_2\}} \delta_{k=l=0,i_1=i_2,j_1=j_2} \log \log N + O(1) \\
    &= (m_1+1)(m_2+1)\log \log N+ O(1) \\
    &= d_\pi\log \log N + O(1).
\end{align*}
Combine these, we have
\begin{align*}
\sum_{p \leq N}\frac{|\lambda_\pi(p)|}{p} &\leq (a_0(d_\pi-1)d_\pi + \frac{a_1(d_\pi-1)}{d_\pi} + \frac{a_2(d_\pi-1)}{d_\pi^2})\log \log N + O(1) \\
&= (1-\eta_{m_1,m_2})\log \log N + O(1)
\end{align*}
where
$$\eta_{m_1,m_2} = \frac{(d_\pi-1)(d_\pi+1)}{(d_\pi-1)^2+3(d_\pi-1)+1}\left(\frac{d_\pi+2}{d_\pi+1} - \sqrt{\frac{d_\pi+1}{d_\pi}}\right).$$
For a detailed calculation, see \cite[Lemma~3.3]{TaWu2016}.
Finally, we have
\begin{align*}
    \sum_{n=1}^N |\mu_\pi(n)| &\ll \frac{N}{\log N} \exp((1-\eta_{m_1,m_2})\log \log N) \\ 
    &= \frac{N}{\log^{\eta_{m_1,m_2}}N}
\end{align*}
where $\eta_{m_1,m_2} > 0$ if and only if $d_\pi = (m_1+1)(m_2+1) > 1$, that is, $m_1$ and $m_2$ not both zero. The proof of Theorem 1.1 is now complete.

\section{Proof of Theorem 1.5}
Before specializing to Maass forms, we are going to establish an inequality analogous to Lemma \ref{Shiu} that works for a boarder class of $L$-functions.

\begin{proposition}\label{non-ramanujan lemma}
Let $\pi$ be a self-dual, everywhere unramified cuspidal automorphic representation of GL$(d,\bA_\bQ)$. Then we have
\[\sum_{n\leq x}|\mu_{\pi}(n)|\ll\frac{x}{\log x}\exp\left(\sum_{p\leq x}\frac{|\mu_{\pi}(p)|}{p}\right).\]
\end{proposition}

\begin{proof}
Our starting point is to apply Lemma \ref{Maass form Shiu} to $g(n)=|\mu_{\pi}(n)|$. We claim the following inequalities hold:
\begin{flalign*}
\sum_{p\leq x}|\mu_{\pi}(p)|\log p&\ll x\\
\sum_{p\leq x}\sum_{\alpha\geq 2}\frac{|\mu_{\pi}(p^{\alpha})|}{p^{\alpha}}\log p^{\alpha}&\ll 1.
\end{flalign*}
(This shows that we can take $h_1(x)=1$ and $h_2(x)=1$ in Lemma \ref{Maass form Shiu}.) For the first inequality, we recall
\[-\frac{L'(s,\pi\otimes\pi)}{L(s,\pi\otimes\pi)}=\sum_{n=1}^{\infty}\frac{\Lambda_{\pi\otimes\pi}(n)}{n^s}, \hspace{3mm} \Re(s) \gg 1,\]
where $\Lambda_{\pi \times \pi}(n)$ is a non-negative function supported on prime powers and its values at primes are given by 
$$\Lambda_{\pi \times \pi}(p)=\lambda_{\pi}(p)^2\log p=\mu_{\pi}(p)^2\log p.$$ 
Then by Rankin-Selberg theory and Theorem 5.13 in \cite{IwKo2004}, we have:
\[\sum_{n\leq x}\Lambda_{\pi\otimes\pi}(n)\sim x, \hspace{3mm} x \to \infty.\]
We then apply Cauchy–Schwarz inequality to get
\[\sum_{p\leq x}|\mu_{\pi}(p)|\log p\leq\left(\sum_{p\leq x}\mu_{\pi}(p)^2\log p\right)^{1/2}\left(\sum_{p\leq x}\log p\right)^{1/2}\ll x.\]
On the other hand, we know that 
$$|\alpha_{\pi,i}(p)| \leq p^{\frac{1}{2}-\eta_d}$$
where $\eta_d = \frac{1}{d^2+1}$ (see \cite{LuRuSa1995}). Hence 
$$|\mu_{\pi}(p^{\alpha})|\ll p^{\alpha(\frac{1}{2}-\eta_d)}$$
for all $\alpha\geq 1.$ We also have $|\mu_{\pi}(p^{\alpha})|=0$ for $\alpha>d.$ In this case,
\[\sum_{p\leq x}\sum_{\alpha\geq 2}\frac{|\mu_{\pi}(p^{\alpha})|}{p^{\alpha}}\log p^{\alpha}\ll \sum_{p\leq x}\sum_{\alpha=2}^{d}\frac{1}{p^{\alpha(\frac{1}{2}+\eta_d)}}\log p^{\alpha}\ll 1.\]
Therefore, Lemma \ref{Maass form Shiu} applies to $g(n) = |\mu_\pi(n)|$ with $h_1(x) = h_2(x) = 1$ and we get
\[\sum_{n\leq x}|\mu_{\pi}(n)|\ll\frac{x}{\log x}\sum_{n\leq x}\frac{|\mu_{\pi}(n)|}{n}.\]
We then proceed to show
\[\sum_{n\leq x}\frac{|\mu_{\pi}(n)|}{n}\ll \exp\left(\sum_{p\leq x}\frac{|\mu_{\pi}(p)|}{p}\right).\]
Since $|\mu_{\pi}(n)|$ is a multiplicative function, we have:
\[\sum_{n\leq x}\frac{|\mu_{\pi}(n)|}{n}\leq \prod_{p\leq x}\left(1+\frac{|\mu_{\pi}(p)|}{p}+\cdots+\frac{|\mu_{\pi}(p^{d})|}{p^{d}}\right).\]
Then by $1+x \leq \exp(x)$ for $x\geq 0,$ we have
\begin{flalign*}
\prod_{p\leq x}\left(1+\frac{|\mu_{\pi}(p)|}{p}+\cdots+\frac{|\mu_{\pi}(p^{d})|}{p^{d}}\right)&\leq\exp\left(\sum_{p\leq x}\sum_{\alpha=1}^{d}\frac{|\mu_{\pi}(p^{\alpha})|}{p^{\alpha}}\right)\\
&=\exp\left(\sum_{p\leq x}\frac{|\mu_{\pi}(p)|}{p}+\sum_{p\leq x}\sum_{\alpha=2}^{d}\frac{|\mu_{\pi}(p^{\alpha})|}{p^{\alpha}}\right).
\end{flalign*}
We already showed that
\[\sum_{p\leq x}\sum_{\alpha=2}^{d}\frac{|\mu_{\pi}(p^{\alpha})|}{p^{\alpha}}\ll1\]
since $|\mu_{\pi}(p^{\alpha})|\ll p^{\alpha(\frac{1}{2}-\eta_d)}.$ This completes the proof of Proposition 4.1.
\end{proof}

We now apply Proposition \ref{non-ramanujan lemma} to the automorphic representation associated to Maass form $\phi$ to give a proof of Theorem \ref{absolute maass}.
\par
By Proposition \ref{non-ramanujan lemma}, it suffices to show
\[\sum_{p\leq x}\frac{|\lambda_{\phi}(p)|}{p}\leq \left(1-\frac{1}{12}\right)\log\log x.\]
We make use of the following inequality in \cite[Equation~(65)]{Ho2009}
\[t^{1/2}\leq 1+\frac{1}{2}(t-1)-\frac{1}{9}(t-1)^2+\frac{1}{36}(t-1)^3,\hspace{3mm} t \geq 0.\]
Let $t=|\lambda_{\phi}(p)|^2$. We have
\[|\lambda_{\phi}(p)|\leq \left(1-\frac{1}{2}-\frac{1}{9}-\frac{1}{36}\right)+\left(\frac{1}{2}+\frac{2}{9}+\frac{1}{12}\right)\lambda_{\phi}(p)^2+\left(-\frac{1}{9}-\frac{1}{12}\right)\lambda_{\phi}(p)^4+\frac{1}{36}\lambda_{\phi}(p)^6.
\]
We have the following relations:
\[\lambda_{\phi}(p)^4=2+3\lambda_{\Sym^2(\phi)}(p)+\lambda_{\Sym^4(\phi)}(p)\]
and
\[\lambda_{\phi}(p)^6=5+8\lambda_{\Sym^2(\phi)}(p)+4\lambda_{\Sym^4(\phi)}(p)+\lambda_{\Sym^2(\phi)}(p)\lambda_{\Sym^4(\phi)}(p)\]
By Rankin-Selberg theory we have
\[\sum_{p\leq x}\frac{|\lambda_{\phi}(p)|^2}{p}=\log\log x+O(1).\]
Use the fact that $\Sym^j(\phi)$ is cuspidal for $j=2,3,4$, we have, by Lemma \ref{orthognality relations for L functions} (a),
\[\sum_{p\leq x}\frac{\lambda_{\Sym^j(\phi)}(p)}{p}=O(1)\]
for $j=2,3,4.$ (Note that we can take $\pi_1=\Sym^j(\phi)$ and $\pi_2$ to be the trivial representation. The trivial representation obviously satisfies the Ramanujan conjecture.)
We can also write that $\lambda_{\Sym^2(\phi)}(p)\lambda_{\Sym^4(\phi)}(p)=\lambda_{\Sym^3(\phi)}(p)^2-1$. By Lemma \ref{orthognality relations for L functions} (b), we have:
\[\sum_{p\leq x}\frac{\lambda_{\Sym^3(\phi)}(p)^2}{p}\leq \log\log x.\]
Combining all the results together, we have:
\[\sum_{p\leq x}\frac{|\lambda_{\phi}(p)|}{p}\leq \left(1-\frac{1}{12}\right)\log\log x\]

\printbibliography

\vspace{5ex}
\noindent Department of Mathematics, Brown University, Providence, RI 02912, USA.

\noindent E-mail address: {\tt zhining\textunderscore wei@brown.edu}

\vspace{2ex}
\noindent Department of Mathematics, The Ohio State University, Columbus, OH 43210, USA.

\noindent E-mail address: {\tt zhao.3326@osu.edu}

\end{document}